\documentclass[12pt,reqno]{amsart}
\usepackage{amssymb}
\usepackage{amsmath}
\usepackage{amsthm}
\usepackage{mathptmx}
\usepackage{color}
\usepackage[pdftex]{graphicx}

\usepackage[notcite,notref]{showkeys}

\newcommand{\R}{\mathbb R}

\newcommand{\Z}{\mathbb Z}

\newcommand{\p}{\partial}

\newcommand{\sgn}{\text{sgn}}

\newcommand{\hil}{\textsf{H}}
\newtheorem{theorem}{Theorem}[section]
\newtheorem{proposition}[theorem]{Proposition}
\newtheorem{remark}[theorem]{Remark}
\newtheorem{lemma}[theorem]{Lemma}
\newtheorem{corollary}[theorem]{Corollary}

\newtheorem*{TA}{Theorem A}
\newtheorem*{TB}{Theorem B}

\numberwithin{equation}{section}

\begin{document}
\title[Regularity propagation BO equation]{On the propagation of regularities in solutions of the Benjamin-Ono equation}
\author{Pedro Isaza}
\address[P. Isaza]{Departamento  de Matem\'aticas\\
Universidad Nacional de Colombia\\ A. A. 3840, Medellin\\Colombia}
\email{pisaza@unal.edu.co}

\author{Felipe Linares}
\address[F. Linares]{IMPA\\
Instituto Matem\'atica Pura e Aplicada\\
Estrada Dona Castorina 110\\
22460-320, Rio de Janeiro, RJ\\Brazil}
\email{linares@impa.br}

\author{Gustavo Ponce}
\address[G. Ponce]{Department  of Mathematics\\
University of California\\
Santa Barbara, CA 93106\\
USA.}
\email{ponce@math.ucsb.edu}

\date{8/8/14}
\begin{abstract} We shall deduce some special regularity properties of  solutions to the IVP associated to the Benjamin-Ono equation.
Mainly, for datum $u_0\in H^{3/2}(\R)$ whose restriction  belongs to $H^m((b,\infty))$ for some $m\in\Z^+,\,m\geq 2,$ and $b\in \R$ we  shall prove that the restriction of the corresponding solution $u(\cdot,t)$ belongs to $H^m((\beta,\infty))$ for any $\beta\in \R$ and any $t>0$. Therefore, this  type of regularity of the datum travels with infinite speed to its left as  time evolves.
\end{abstract}

\maketitle
\section{Introduction}

This work is mainly concerned with special properties of solutions to the initial value problem (IVP) associated to the Benjamin-Ono (BO) equation
\begin{equation}\label{bo1}
\begin{cases}
\partial_t u-\partial_x^2\hil u+u\,\partial_x u=0, \hskip10pt x,t\in\R, \\
u(x,0)=u_0(x),
\end{cases}
\end{equation}
where $\hil$ denotes the Hilbert transform,
\begin{equation}
\label{hita}
\begin{split}
\hil f(x)&=\frac{1}{\pi} {\rm p.v.}\big(\frac{1}{x}\ast f\big)(x)\\
&=\frac{1}{\pi}\lim_{\epsilon\downarrow 0}\int\limits_{|y|\ge \epsilon} \frac{f(x-y)}{y}\,dy=(-i\,\sgn(\xi) \widehat{f}(\xi))^{\vee}(x).
\end{split}
\end{equation}

 The BO equation was first deduced by Benjamin
\cite{Be} and Ono \cite{On} as  a model for long internal gravity
waves in deep stratified fluids. Later, it was also shown to be   a
completely integrable system (see \cite{AbFo}, \cite{CoWi} and
references therein). Thus,  solutions of the IVP \eqref{bo1} satisfy
infinitely many conservation laws which provides an \it a priori \rm
estimate for the $H^{n/2}$-norm, $n\in \Z^+$, of the solution
$\,u=u(x,t)\,$ for \eqref{bo1}. Here we shall only consider  real
valued solutions of the IVP \eqref{bo1}.

 Following the definition in \cite{Ka} it is said that the IVP \eqref{bo1} is locally well-posed (LWP) in the function space $X$ if given any datum $u_0\in X$ there exists $T>0$ and a unique solution
$$
u\in C([-T,T] :X)\cap....
$$
of the IVP \eqref{bo1} with the map data-solution, $u_0\mapsto u$, being continuous. If $T$ can be taken arbitrarily large, one says that
the     IVP \eqref{bo1} is globally well posed (GWP) in $X$.

The problem of finding the minimal regularity property, measured in the Sobolev scale
\begin{equation*}
H^s(\R)= \left(1-\partial^2_x\right)^{-s/2} L^2(\R),\;\;\;\;\,s\in\R,
\end{equation*}
 required to guarantee that the IVP  \eqref{bo1}  is locally or globally well-posed in $\,H^s(\R)$  has been extensively studied. Thus, one has the
following list of works  \cite{ABFS}, \cite{Io1}, \cite{Po}, \cite{KoTz1}, \cite{KeKo}, \cite{Ta}, \cite{BuPl} and
 \cite{IoKe} where GWP was established in $H^0(\R)=L^2(\R)$, (for further details  and results regarding the well-posedness of the IVP \eqref{bo1} in $H^s(\R)$  we refer to \cite{MoPi}). It should be pointed out that a result found in \cite{MoSaTz}
(see also \cite{KoTz2})
implies that none well-posedness in $H^s(\R)$ for any $\,s\in\R\,$ for the IVP \eqref{bo1} can be established by a
solely contraction principle argument.

Well-posedness of the IVP \eqref{bo1} has been also studied in  weighted Sobolev spaces
\begin{equation}
\label{spaceZ}
Z_{s,r}=H^s(\R)\cap L^2(|x|^{2r}dx),\;\;\;\;\;\;s,\,r\in\R,
\end{equation}
and
\begin{equation}
\label{spaceZdot}
\dot Z_{s,r}=\{ f\in H^s(\R)\cap L^2(|x|^{2r}dx)\,:\,\widehat
{f}(0)=0\},\;\;\;\;\;\;\;s,\,r\in\R.
\end{equation}
 Notice that the conservation law for solutions of \eqref{bo1}
 $$
 I_1(u_0)=\int_{-\infty}^{\infty} u(x,t)dx=\int_{-\infty}^{\infty} u_0(x)dx,
 $$
 guarantees that the property $\,\widehat{u}_0(0)=0$
 is preserved by the solution flow.

  Motivated by the results  in \cite{Io1}, \cite{Io2}  it was established in
  \cite{FP} that the IVP \eqref{bo1} is :\newline LWP (resp. GWP) in $Z_{s,r}$ for $s>9/8$ (resp. $s\geq 3/2$), $s\geq r$ and $r\in(0,5/2)$,
  \newline and
  \newline
GWP in $\dot Z_{s,r}$ with  $s\geq r$ and $r\in [5/2,7/2)$.
\newline
Moreover, it was also shown in \cite{FP} that the above
results are optimal (for further details and results concerning the
well posedness of the IVP \eqref{bo1} in weighted Sobolev spaces  we
refer to \cite{FLP}).

 Although we shall be mainly engaged with the IVP \eqref{bo1} our results below will also apply to solutions of the IVP associated to the generalized Benjamin-Ono (g-BO) equation
\begin{equation}\label{bok}
\begin{cases}
\partial_t u-\partial_x^2\hil u+u^k\,\partial_x u=0, \hskip10pt x,t\in\R, \;\;k\in \Z^{+},\\
u(x,0)=u_0(x).
\end{cases}
\end{equation}

In this case well-posedness of the IVP for $k\geq 2$ has been considered in \cite{KPV94}, \cite{MoRi1}, \cite{MoRi2}, \cite{KeTa} where GWP for the IVP \eqref{bok} with $k=2$ was obtained in $H^s(\R),\;s\geq 1/2$ and  \cite{Vent} where LWP of the IVP \eqref{bok} was proven in the critical Sobolev space $H^{s_k}(\R)$ with $s_k=1/2-1/k$ for $k\geq 4$ and with  $s>1/3$ for $k=3$.
 Addressing the well-posedness of the IVP \eqref{bok} in the weighted Sobolev spaces \eqref{spaceZ} and \eqref{spaceZdot} one has that the positive results state before for the IVP \eqref{bo1} also applies. However, the optimality of these results in the cases where the power $k$ in \eqref{bok} is an even integer has not been established.

To state our main result we need to describe the space of solutions where we shall be working on.

\begin{TA}{\rm(\cite{Po})}\label{gwp}
If $u_0\in H^{s}(\R)$ with $\,s\geq 3/2$, then there exists  a unique  solution $u=u(x,t)$ of the IVP \eqref{bo1}
such that for any $T>0$
\begin{equation}\label{bo2}
\begin{split}
{\rm (i)}\hskip15pt & u\in C(\R : H^{s}(\R))\cap L^{\infty}(\R : H^{s}(\R)),\\
{\rm(ii)}\hskip15pt & \partial_x u\in L^4([-T,T]: L^{\infty}(\R)), \hskip10pt {\text (Strichartz)},\\
{\rm(iii)}\hskip15pt &\int\limits_{-T}^{T}\int\limits_{-R}^{R}  (|\partial_x D_xu|^2+|\partial_x^2u|^2)(x,t)|^2dxdt\le c_0,\\
{\rm(iv)}\hskip15pt &\int\limits_{-T}^{T}\int\limits_{-R}^{R}  |\partial_x D_x^{r-1/2}u(x,t)|^2dxdt\le c_1,\;\;\;r\in[1,s],
\end{split}
\end{equation}
with $\,c_0=c_0(R, T, \|u_0\|_{3/2,2})\,$ and $\,c_1=c_1(R, T, \|u_0\|_{s,2})$.
\end{TA}

 \begin{remark} From our previous comments it is clear that the results in Theorem A still holds for the IVP \eqref{bok} with $k=2$
(see \cite{KeKo}) and locally in time for $k\geq 3$. Indeed, one can
lower the requirement $s\geq 3/2$ to a value between $[1, 3/2]$,
depending on the $k$ considered, such that  a well-posedness
including the estimate \eqref{bo2} (ii) still holds. However, we
will not consider this question here.
\end{remark}

 To state our main result we introduce the two parameter family $(\epsilon, b)$ with $\epsilon>0, \,b\geq 5\epsilon$ of functions $\chi_{\epsilon,b}\in C^{\infty}(\R)$ such that $\chi'_{\epsilon,b}(x)\geq 0$ with
\begin{equation}
\label{1.1}
\chi_{\epsilon,b}(x)=\begin{cases}
0,\hskip10pt x\le \epsilon,\\
1, \hskip 10pt x\ge b,
\end{cases}
\end{equation}
satisfying \eqref{2.2}-\eqref{2.8} in section 2.

\begin{theorem}\label{Th1}
Let $u_0\in H^{3/2}(\R)$ and $u=u(x,t)$ be the corresponding solution of the IVP \eqref{bo1} provided by  Theorem A. If for some $x_0\in\R$ and for some $m\in\Z^{+}$, $m\ge 2$,
\begin{equation}\label{bo3}
\int\limits_{x_0}^{\infty} (\partial_x^m u_0)^2(x)\,dx <\infty,
\end{equation}
then for any $v>0$, $T>0$, $\epsilon>0$, $b\geq 5\epsilon$
\begin{equation}\label{bo4}
\begin{split}
&\underset{0\le t\le T}{\sup} \,\int (\partial_x^m u(x,t))^2\,\chi_{\epsilon,b}(x-x_0+vt)\,\,dx\\
&+\int\limits_0^T\int (D_x^{1/2}\partial_x^mu(x,t))^2 \chi_{\epsilon,b}'(x-x_0+vt)\,\,dxdt<c=c(T, \epsilon, b, v).
\end{split}
\end{equation}
\vskip.1in

If  in addition to \eqref{bo3}  there exists  $x_0\in \R$ such that
any $\epsilon>0$, $b>5\epsilon$
\begin{equation}\label{bo5}
D^{1/2}_x(\partial_x^m u_0\,\chi_{_{\epsilon, b}}(\cdot-x_0)) \in L^2(\R),
\end{equation}
then
\begin{equation}\label{bo6}
\begin{split}
&\underset{0\le t\le T}{\sup} \int (D_x^{1/2}(\partial_x^mu(x,t)\,\chi_{_{\epsilon, b}}(x-x_0+vt)))^2\,dx\\
&+\int\limits_0^T\int (\partial_x^{m+1}u(x,t))^2 \,\chi_{_{\epsilon, b}}'\,\chi_{_{\epsilon, b}}(x-x_0+vt)\,dxdt<c,
\end{split}
\end{equation}
\end{theorem}
with $c=c(T, \epsilon, b, v)$.

\begin{theorem}\label{Th2}
With the same hypotheses the results in Theorem \ref{Th1} apply to
solutions of the IVP \eqref{bok} globally in time if $\,k=2$, and
locally in  time if  $\,k\geq 3$.
\end{theorem}

\begin {remark}  \hskip10pt

{\rm(a)} From our comments above  and our proof of Theorem \ref{Th1} it will be clear that the requirement $u_0\in H^{3/2}(\R)$ in Theorem \ref{Th1} can be lowered to $u_0\in H^{9/8^+}(\R)$ by considering the problem  in a finite time interval (see \cite{KeKo}).

{\rm (b)} It will be clear from our proofs of Theorem \ref{Th1} and Theorem \ref{Th2} that they still hold for solutions of the \lq\lq defocussing" gBO equation
\begin{equation*}
\partial_t u- H \partial_x^2u-u^k\,\partial_x u=0, \hskip10pt x,t\in\R,\;\;\;k\in \Z^{+}.
\end{equation*}
Therefore, our results apply to $\,u(-x,-t)\,$ if $\,u(x,t)\,$ is a solution of \eqref{bok}. In other words, Theorem \ref{Th1} and Theorem \ref{Th2} resp. remain valid,
backward in time, for datum satisfying the hypothesis \eqref{bo3} and \eqref{bo5} respectively  on the left hand side of the real line. 
\end{remark}

The following are direct outcomes of the above comment, the group properties,  and Theorems \ref{Th1} and
\ref{Th2}.  In order to simplify the exposition we shall state them
only for solutions of the IVP \eqref{bo1}.  First, as a direct
consequence of Theorem \ref{Th1}  and the time reversible character
of the equation in \eqref{bo1} one has :

\begin{corollary}\label{cor.1}
Let $\,u\in C(\R : H^{{3/2}^{+}}(\R))$ be a solution of the equation in \eqref{bo1} described in Theorem A.
 If there exist $ \,m\in\Z^{+},\,m\geq 2,\,\hat t\in\R,\;a\in\R $ such that
\begin{equation*}
\partial_x^mu(\cdot,\hat t)\notin L^2((a,\infty)),
\end{equation*}
then for any $t\in (-\infty,\hat t)$ and any $\beta\in\R$

\begin{equation*}
\partial_x^mu(\cdot,t)\notin L^2((\beta,\infty)).
\end{equation*}
\end{corollary}

 Next, one has  that for appropriate class of data  singularities of the corresponding solutions travel with infinite speed to the left as time evolves.

\begin{corollary}\label{cor.2}
Let $\,u\in C(\R: H^{3/2}(\R))$ be a solution of the equation in \eqref{bo1} described in Theorem A.
 If there exist  $ \,k, m \in\Z^{+}$ with  $k\ge m$ and $a,\, b\in \R$ with $\,b<a$ such that
\begin{equation}\label{A}
\int_{a}^{\infty} |\partial_x^ku_0(x)|^2dx <\infty\;\;\;\text{but}\;\;\;\partial_x^m u_0\notin L^2((b,\infty)),
\end{equation}
then for any $t\in(0,\infty)$ and any $v>0$ and $\epsilon>0$
\begin{equation*}
\int_{a+\epsilon-vt}^{\infty} |\partial_x^ku(x,t)|^2dx <\infty,
\end{equation*}
and for any $t\in(-\infty,0)$ and  $\alpha\in\R$
\begin{equation*}
\int_{\alpha}^{\infty} |\partial_x^m u(x,t)|^2dx =\infty.
\end{equation*}
\end{corollary}

\begin{remark}\hskip10pt

{\rm(a)}  If in Corollary \ref{cor.2} in addition to  \eqref{A} one assumes that
$$
\int_{-\infty}^b |\partial_x^k u_0(x)|^2dx<\infty,
$$
then by combining the results in this corollary  with the group properties
it follows that
\begin{equation*}
\int_{-\infty}^{\beta} |\partial_x^m u(x,t)|^2\,dx =\infty, \hskip10pt \text{for any}\hskip 5pt \beta\in\R\hskip5pt\text{and}\hskip5pt t>0.
\end{equation*}
This shows that the regularity in the left hand side does not propagate forward in time.

{\rm(b)} Notice that \eqref{bo4} implies: for any $\epsilon>0$, $v>0$, $T>0$
\begin{equation}\label{bo7}
\underset{0\le t\le T}{\sup}\, \int_{x_0+\epsilon-vt}^{\infty}(\partial_x^k u(x,t))^2\,dx \le c = c( \epsilon,v,T).
\end{equation}

 This tells us that the local regularity of the initial datum $u_0$ described in \eqref{bo3} propagates with infinite speed to its left as time evolves.

{\rm(c)} In \cite{ILP1} we proved the corresponding result concerning the IVP for the $k$-generalized Korteweg-de Vries equation

\begin{equation}\label{gkdv}
\begin{cases}
\partial_t u+\partial_x^3u+u^k\,\partial_x u=0, \hskip10pt x,t\in\R, \;\;k\in \Z^{+},\\
u(x,0)=u_0(x).
\end{cases}
\end{equation}

More precisely, the following result was obtained in \cite{ILP1}:

\begin{TB}
If  $u_0\in H^{{3/4}^{+}}(\R)$ and for some $\,m\in \Z^{+},\,\;m\geq 1$ and $x_0\in \R$
\begin{equation}\label{notes-3}
\|\,\partial_x^m u_0\|^2_{L^2((x_0,\infty))}=\int_{x_0}^{\infty}|\partial_x^m u_0(x)|^2dx<\infty,
\end{equation}
then the solution of the IVP \eqref{gkdv} satisfies  that for any $v>0$ and $\epsilon>0$
\begin{equation}\label{notes-4}
\underset{0\le t\le T}{\sup}\;\int^{\infty}_{x_0+\epsilon -vt } (\partial_x^j u)^2(x,t)\,dx<c,
\end{equation}
for $j=0,1, \dots, m$ with $c = c(l,  v, \epsilon, T)$.

Moreover, for any $v\geq 0$, $\epsilon>0$ and $R>\epsilon$
\begin{equation}\label{notes-5}
\int_0^T\int_{x_0+\epsilon -vt}^{x_0+R-vt}  (\partial_x^{j+1} u)^2(x,t)\,dx dt< c,
\end{equation}
with  $c = c(l, v, \epsilon, R, T)$.
\end{TB}

However, the proof for the BO equation considered here is quite more
involved. First, it includes a non-local operator, the Hilbert
transform \eqref{hita}. Second, in the case of the $k$-gKdV the
local smoothing effect yields a gain of one derivative which allows
to pass to the next step in the inductive process. However, in the
case of the BO equation the gain of the local smoothing is just
$1/2$-derivative so the iterative argument has to be  carried out
in two steps, one for positive integers $m$  and another one for
$m+1/2$. Also the explicit identity obtained in \cite{Ka} describing
the local smoothing effect in solutions of the KdV equation is not
available for the BO equation. In this case, to establish the local
smoothing one has to rely on several commutator estimates. The main
one is the extension of the Calder\'on first commutator estimate for
the Hilbert transform \cite{Ca} given by Bajvsank and Coifman  in
\cite{BaCo} (see  Theorem \ref{Cal} in section 2).

{\rm (d)}  Without loss of generality from now on we shall assume that in Theorem \ref{Th1} $x_0=0$.

{\rm (e)} We recall that the above result still hold if one replaces $x, t>0$ by  $x, t<0$.
\end{remark}

 The rest of this paper is organized as follows: section 2 includes the description of the two parameter family of cut-off functions
 to be  employed in the proof of Theorem \ref{Th1}. Also section 2 has the statements and some proofs of  the commutator estimates
 and the  interpolation inequalities  needed in the proof of Theorem \ref{Th1}. The proof of Theorem \ref{Th1} will be
 given in section 3.
 
\section{Preliminaries}

For each $\epsilon>0$ and $b\ge 5\epsilon$ we define a function
$\chi_{_{\epsilon, b}} \in C^{\infty}(\R)$ with  $\;\chi_{_{\epsilon, b}}'(x)\ge0$, and
\begin{equation}
\label{2.1}
\chi_{\epsilon,b}(x)=\begin{cases}
0,\hskip10pt x\le \epsilon,\\
1, \hskip 10pt x\ge b,
\end{cases}
\end{equation}
which will be constructed as follows. Let $\rho\in C^{\infty}_0(\R)$, $\rho(x)\geq 0$, even, with $\,\text{supp} \,\rho\subseteq(-1,1)$ and $\,\int\,\rho(x)dx=1$ and define
\begin{equation}
\label{2.2}
\nu_{_{\epsilon,b}}(x)=\begin{cases}
0,\hskip10pt x\le 2\epsilon,\\
\\
\frac{1}{b-3\epsilon} x-\frac{2\epsilon}{b-3\epsilon},\;\;\;\,x\in[2\epsilon,b-\epsilon],\\
\\
1, \hskip 10pt x\ge b-\epsilon,
\end{cases}
\end{equation}
and
\begin{equation}
\label{2.3}
\chi_{_{ \epsilon, b}}(x)=\rho_{\epsilon}\ast \nu_{\epsilon,b}(x)
\end{equation}
where $\rho_{\epsilon}(x)=\epsilon^{-1}\rho(x/\epsilon)$. Thus
\begin{equation}
\label{2.4}
\begin{split}
& \text{supp}\;\chi_{_{ \epsilon, b}}\subseteq [\epsilon,\infty),\\
& \text{supp}\; \chi_{_{ \epsilon, b}}'(x)\subseteq [\epsilon, b].
\end{split}
\end{equation}

If $\;x\in(3\epsilon, b-2\epsilon)$, then
\begin{equation}
\label{2.5}
\chi_{_{\epsilon, b}}'(x)\geq \frac{1}{b-3\epsilon}.
\end{equation}

 If $\;x\in (3\epsilon, \infty)$, then
\begin{equation}
\label{2.6}
\chi_{_{\epsilon, b}}(x)\geq \chi_{_{\epsilon, b}}(3\epsilon)\geq \frac{1}{2}\,\frac{\epsilon}{b-3\epsilon},
\end{equation}
and for any $\,x\in \R$
\begin{equation}
\label{2.7}
\chi_{_{\epsilon, b}}'(x) \leq \frac{1}{b-3\epsilon}.
\end{equation}

Now we define $\eta_{_{\epsilon, b}}$ by the identities
\begin{equation}
\label{2.8}
\chi_{_{\epsilon, b}}'(x)=\eta_{_{\epsilon, b}}^2, \;\;\;\;\hskip5pt{i.e.}\hskip5pt \;\;\;\;\;\eta_{_{\epsilon, b}}=\sqrt{\chi_{_{\epsilon, b}}'(x)}.
\end{equation}
\vskip.1in
\noindent\underline{CLAIM} : For any  $\epsilon>0$ and $b\ge 5\epsilon$ $\;\eta_{\epsilon, b}\in C^{\infty}_0(\R)$ with $\text{supp}\; \eta_{\epsilon, b} =\text{supp}\;\chi'_{\epsilon, b}$.
\vskip.1in
It suffices to show that if $f\in C_0^{\infty}(\R)$ with $\text{supp}\;f=[0,1]$ and $f(x)>0$, $x\in(0,1)$ then $\sqrt{f}$ is smooth at $x=0$ and $x=1$.
Without loss of generality we only consider the case $x=0$.

By hypothesis on $f$  for every $n\in \Z^{+}$ there exist $M_n\ge 0$ and $\delta_n>0$ such that
\begin{equation}\label{ast1}
|f(x)|\le M_n\,|x|^n, \hskip10pt |x|\le \delta_n.
\end{equation}
Thus
\begin{equation*}
\frac{d}{dx}\big(\sqrt{f}\big)(0)=\underset{h\to 0}{\lim}\, \frac{\sqrt{f(h)}-0}{h}=0
\end{equation*}
by \eqref{ast1}. Also
\begin{equation*}
\frac{d^2}{dx^2}\big(\sqrt{f}\big)(0)=\underset{h\to 0}{\lim}\, \frac{\sqrt{f(h)}-2\sqrt{f(0)}+\sqrt{f(-h)}}{h}=0
\end{equation*}
by \eqref{ast1}. Following this argument of writing the derivatives of $\sqrt{f}$ using finite differences, from \eqref{ast1} one obtains  the desired result, i.e.
\begin{equation*}
\frac{d^n}{dx^n}\big(\sqrt{f}\big)(0)= 0 \hskip10pt \text{for any} \hskip5pt n,
\end{equation*}
so $\,\sqrt{f}\,$ is smooth.
\vskip.2in
We shall also use that given $\epsilon>0$, $b\ge 5\epsilon\,$ there exists $c=c_{\epsilon, b}>0$ such that
\begin{equation}\label{CL}
\begin{split}
& \chi_{_{\epsilon/5,\epsilon}}(x)=1,\;\;\;\;\;\text{ on supp } \,\chi_{_{\epsilon,b}},\\
& \chi_{_{\epsilon,b}}'(x) \leq c\, \chi_{_{\epsilon/3,b+2\epsilon/3}}'(x)\, \chi_{_{ \epsilon/3, b+2\epsilon/3}}(x),\\
&\chi_{_{\epsilon,b}}'(x)   \leq c\,\chi_{_{\epsilon/5,\epsilon}}(x).
\end{split}
\end{equation}

The following extension of the Calder\'on commutator theorem \cite{Ca} established by Bajvsank and Coifman in \cite{BaCo} will be a crucial ingredient in our proof of Theorem \ref{Th1}.

\begin{theorem}
\label{Cal}
Let $\hil$ be the Hilbert transform. For any $p\in (1,\infty)$ and any $\, l, m \in \Z^{+}$, $\; l+m\ge 1$ there exists
$c=c(p; l; m)>0$ such that
\begin{equation}\label{CE}
\|\partial_x^l [\hil; \psi]\,\partial_x^m f\|_p\le c\,\|\partial_x^{m+l}\psi\|_{\infty} \|f\|_p.
\end{equation}
\end{theorem}

For a different proof of this estimate see Lemma 3.1 in \cite{DaMcPo}. Also we shall use the following commutator estimate:

\begin{proposition}
\begin{equation}\label{CE2}
\|[D_x^{1/2}; h]\partial_x f\|_2\le c\,\|\widehat{\partial_x h}\|_1\|D_x^{1/2}f\|_2.
\end{equation}
\end{proposition}

\begin{proof}

 Taking Fourier transform it follows that
\begin{equation}
\label{11}
\widehat{[D^{1/2}; h]\partial_x f}(\xi)=c\int (|\xi|^{1/2}-|\eta|^{1/2})\,\eta \,\widehat{h}(\xi-\eta)\,\widehat{f}(\eta)\,d\eta.
\end{equation}

Using the mean value theorem it is easy to see that there exists $c>0$ such that for any $\xi,\,\eta\in \R$
\begin{equation}
\label{12}
|\,|\xi|^{1/2}-|\eta|^{1/2}|\,|\eta|\leq c\,|\eta|^{1/2}\,|\xi-\eta|.
\end{equation}

Therefore, inserting \eqref{12} into \eqref{11} and using Plancherel identity and Young's inequality we obtain the desired result \eqref{CE2}.

\end{proof}

Next, we collect some inequalities concerning the Leibniz rule for fractional derivatives established in \cite{KP}, \cite{KPV}, \cite{GO}
to obtain :

\begin{lemma}  For $\alpha\in (0,1)$, $p\in [1, \infty)$
\begin{equation}\label{FD}
\|D_x^{\alpha}(fg)\|_p \le c\,\big( \|f\|_{p_1}\|D_x^{\alpha}g\|_{p_2}+\|D_x^{\alpha}f\|_{p_3} \|g\|_{p_4}\big)
\end{equation}
with
\begin{equation*}
\frac{1}{p}=\frac{1}{p_1}+\frac{1}{p_2}=\frac{1}{p_3}+\frac{1}{p_4}, \hskip10pt p_j\in(1,\infty],\;\; j=1, 2, 3, 4.
\end{equation*}
\end{lemma}

Finally, we state some interpolation inequalities to be used in the
proof of Theorem \ref{Th1}.
\begin{lemma}

\begin{equation}\label{INT}
\begin{split}
\|f\|_4 &\le c\|D_x^{1/4}f\|_2\le c \|D_x^{1/2}f\|_2^{1/2}\|f\|_2^{1/2},\\
\|D_x^{1/2}f\|_4 &\le c\|\partial_x f\|_4^{1/2}\|f\|_4^{1/2},\\
\|D_x^{1/2}f\|_4 &\le c \|D_x^{3/4}f\|_2\le c \|\partial_x f\|_2^{3/4}\|f\|_2^{1/4}.
\end{split}
\end{equation}
\end{lemma}

For the proof of  \ref{INT} we refer to \cite{BeLo}.

\section{Proof of Theorem A}

We shall use induction in $m\in \Z^{+}$ with $\,m\ge 2$. First we consider the case $m=2$.
\vskip3mm

\noindent\underline{Case $m=2$ in \eqref{bo3} }: Taking second derivative in the equation \eqref{bo1} multiplying the result by $\p_x^2u(x,t)\chi_{_{\epsilon,b}}(x+vt)$ and
integrating in the $x$-variable one gets the identity (omitting the indices $\epsilon, b$ on $\chi$)
\begin{equation}\label{bo8}
\begin{split}
&\frac12 \frac{d}{dt} \int (\p_x^2u)^2\,\chi\,dx-\underset{A_1}{\underbrace{v\int(\p_x^2u)^2\,\chi'\, dx}}\\
&-\underset{A_2}{ \underbrace{\int \hil \p_x^4u\, \p_x^2u\,\chi\,dx}}
+\underset{A_3}{\underbrace{\int \p_x^2(u\p_x u)\,\p_x^2 u \,\chi\, dx}}=0.
\end{split}
\end{equation}

Since given $T>0$, $\epsilon>0$, $b>5\epsilon$, $v>0$, there exist $c>0$ and $R>0$ such that
\begin{equation}\label{bo9}
\chi_{_{\epsilon, b}}'(x+vt)\le c\, 1_{[-R, R]}(x) \hskip10pt \text{for all \hskip5pt} (x,t) \in \R\times [0,T],
\end{equation}
using \eqref{bo2} (iii) one has after integration in the time interval $[0,T]$ that
\begin{equation}\label{bo10}
\begin{split}
\int\limits_0^T &|A_1(t)|\,dt \\
&\le c\,|v|\int\limits_{-T}^{T}\int\limits_{-R}^{R} (\p_x^2u)^2(x,t)\,dxdt
\le c=c(v, R, T, \epsilon, b, \|u_0\|_{3/2,2}).
\end{split}
\end{equation}

To control the contribution of $A_2$ in \eqref{bo8} we write after integration by parts that
\begin{equation}\label{bo11}
A_2(t)=\int \hil \p_x^3u \,\p_x^3u\,\chi \,dx +\int \hil \p_x^3u \,\p_x^2u\,\chi' \,dx= A_{21}(t)+ A_{22}(t).
\end{equation}

Now since the Hilbert transform is skew symmetric
\begin{equation}\label{bo12}
\begin{split}
A_{21}&= -\int \p_x^3 u\,\hil(\p_x^3u\,\chi) \,dx\\
&=-\int \p_x^3 u \,\hil\p_x^3u\,\chi\,dx-\int \; \partial_x^3u\, [\hil; \chi]\,\p_x^3u\,dx\\
&=-A_{21} -\int \p_x^3 u\,[\hil; \chi]\,\p_x^3 u\,dx.
\end{split}
\end{equation}
Therefore
\begin{equation}\label{bo13}
A_{21}= -\frac12\int \p_x^3 u\,[\hil; \chi]\,\p_x^3 u\,dx=\frac12\int u\,\p_x^3 [\hil; \chi]\,\p_x^3u\,dx
\end{equation}
and consequently using the commutator estimate \eqref{CE} it follows that
\begin{equation}\label{bo14}
\underset{0\le t\le T}{\sup} |A_{21}(t)| \le c\,\|u(t)\|_2^2=c\,\|u_0\|_2^2.
\end{equation}

Next using that $\eta^2=\chi'$ we rewrite $A_{22}$ in \eqref{bo11} as
\begin{equation}\label{bo15}
\begin{split}
A_{22}&=\int \hil\p_x^3u\,\eta\,\p_x^2u\,\eta\,dx\\
&=\int \hil(\p_x^3u\,\eta)\,\p_x^2u\,\eta\,dx-\int[\hil; \eta]\,\p_x^3u\,\p_x^2u\,\eta\,dx\\
&=\int \hil\p_x(\p_x^2u\,\eta)\,\p_x^2u\,\eta\,dx-\int\hil(\p_x^2u\,\eta')\,\p_x^2u\,\eta\,dx\\
&\hskip10pt -\int [\hil; \eta]\p_x^3u\,\p_x^2u\,\eta\,dx\\
&=\int D_x^{1/2}(\p_x^2u\,\eta)\,D_x^{1/2}(\p_x^2u\,\eta)\,dx-\int\hil(\p_x^2u\,\eta')\,\p_x^2u\,\eta\,dx\\
&\hskip10pt -\int [\hil; \eta]\p_x^3u\,\p_x^2u\,\eta\,dx
= A_{221} + A_{222} + A_{223}.
\end{split}
\end{equation}

To estimate $\,A_{222}$ in \eqref{bo15} we write
\begin{equation}\label{bo16}
\begin{split}
\int\limits_0^T |A_{222}(t)|\,dt &\le \int\limits_0^T \|\p_x^2u\,\eta'\|_2 \|\p_x^2u\,\eta\|_2\,dt\\
&\le  \int\limits_0^T\int  ( \p_x^2u\,\eta')^2\,dxdt +\int\limits_0^T\int  ( \p_x^2u\,\eta)^2\,dxdt\\
&\le c = c (v, R, T, \epsilon, b, \|u_0\|_{3/2,2}),
\end{split}
\end{equation}
by \eqref{bo2} (iii).

Now, after integration by parts we have that
\begin{equation}\label{bo17a}
A_{223}(t)=\int\,\partial_x[\hil;\eta]\partial_x^3u\,\partial_xu\eta dx+\int [\hil;\eta]\partial_x^3u\,\partial_xu\eta' dx.
\end{equation}

Using the commutator estimate\eqref{CE}, \eqref{bo2} (iii) and \eqref{bo9}, after integrating in the time interval $[0,T]$ one finds that
\begin{equation}\label{bo17}
\begin{split}
\int\limits_0^T |A_{223}(t)&|\,dt \le c\,\int\limits_0^T\|u(t)\|_2( \|\p_x u\,\eta\|_2+\|\partial_x u\eta'\|_2)\,dt\\
&\le c T\underset{0\le t \le T} {\sup} \|u(t)\|_2^2+ c\int\limits_0^T\int (|\p_x u\eta|^2+|\partial_xu\eta'|^2)dxdt\\
&\le c\,T \|u_0\|_2^2 +  c (v, R, T, \epsilon, b, \|u_0\|_{3/2,2}).
\end{split}
\end{equation}

Next,  we observe that $A_{221}(t)$ in \eqref{bo15} is positive and represents the smoothing effect.

Finally we consider $A_3$ in \eqref{bo8}. So one has that
\begin{equation}\label{bo18}
\begin{split}
A_3(t) &= \int u\,\p_x^3u \p_x^2u \,\chi\,dx  +3  \int \p_xu\,\p_x^2u \p_x^2u \,\chi\,dx \\
&= -\frac12  \int u\,\p_x^2u \p_x^2u \,\chi'\,dx +\frac52  \int \p_x u\p_x^2u \p_x^2u \,\chi\,dx \\
&= A_{31}(t)+ A_{32}(t).
\end{split}
\end{equation}

Then
\begin{equation}\label{bo19}
|A_{31}(t)| \le \|u(t)\|_{\infty}\int (\p_x^2u)^2\,\chi'\,dx,
\end{equation}
so after integration in the time interval $[0,T]$ a combination of the Sobolev embedding, \eqref{bo2} (i) and \eqref{bo2} (iii) yields
\begin{equation}\label{bo20}
\begin{split}
\int\limits_0^T |A_{31}(t)|\,dt&\le \underset{0\le t\le T}{\sup}\|u(t)\|_{1,2} \int\limits_0^T\int  (\p_x^2u)^2\,\chi'\,dxdt\\
&\le c (v, R, T, \epsilon, b, \|u_0\|_{3/2,2}).
\end{split}
\end{equation}
Also one gets that
\begin{equation}\label{bo21}
|A_{32}(t)| \le \|\p_x u(t)\|_{\infty} \int (\p_x^2u)^2(x,t)\, \chi_{_{\epsilon, b}}(x+vt)\, dx.
\end{equation}

Inserting the above information in \eqref{bo8}, using Gronwall's inequality and \eqref{bo2} (ii) we obtain \eqref{bo4}
with $m=2$ which is the desired result.

\vskip3mm

\noindent\underline{Case $m=2\;$ in \eqref{bo5}} :  Assuming that \eqref{bo4} holds for $m=2$ and that  \eqref{bo5} holds for any $\epsilon>0$ and $b\geq 5\epsilon$ with $x_0=0$ and $m=2$ we shall prove \eqref{bo6}.

From the previous step, $m=2$,  we know that the solution satisfies that for any $\epsilon>0$, $b> 5\epsilon$, $v>0$ and $T>0$
\begin{equation}\label{bo22}
\begin{split}
&\underset{0\le t\le T}{\sup} \int (\p_x^2 u(x,t))^2\,\chi_{_{\epsilon, b}}(x+vt)\,dx\\
&\hskip15pt+ \int\limits_0^T\int (D_x^{1/2}(\p_x^2u(x,t)\,\eta_{_{\epsilon, b}}(x+vt))^2\,dxdt \le c = c(T,\epsilon, b, v).
\end{split}
\end{equation}
where $\;\chi_{_{\epsilon, b}}'(x) = \eta_{_{\epsilon, b}}^2(x)$.

From the equation \eqref{bo1} one gets that
\begin{equation}\label{bo23}
\p_t(\p_x^2 u\, \chi_{_{\epsilon, b}})-v \p_x^2u\, \chi_{_{\epsilon, b}}' -\p_x^2\hil\p_x^2u\, \chi_{_{\epsilon, b}}+ \p_x^2(u\p_xu)\,\chi_{_{\epsilon, b}}=0.
\end{equation}

Therefore applying $D^{1/2}_x$ to \eqref{bo23}, multiplying the result by $D^{1/2}_x(\p_x^2u\,\chi_{_{\epsilon, b}})$ and integrating in the space variable
it leads to the identity (omitting the indices $\epsilon$ and $b$ on $\chi$)
\begin{equation}\label{bo24}
\begin{split}
&\frac12 \frac{d}{dt} \int (D^{1/2}_x(\p_x^2 u\, \chi))^2\,dx -
\underset{A_1}{\underbrace{v\int D^{1/2}_x(\p_x^2u\, \chi')\,D^{1/2}_x(\p_x^2u\, \chi)\,dx}} \\
&-\underset{A_2}{\underbrace{\int D^{1/2}_x(\hil\p_x^4u\, \chi)\,D^{1/2}_x(\p_x^2u\, \chi)\,dx}}\\
&+\underset{A_3}{\underbrace{\int D^{1/2}_x( \p_x^2(u\p_xu)\,\chi)\,D^{1/2}_x(\p_x^2u\, \chi)\,dx}} = 0.
\end{split}
\end{equation}

First we have
\begin{equation}\label{bo25}
\begin{split}
|A_1(t)|&\le  |v|\int (D^{1/2}_x(\p_x^2u\,\chi'))^2\,dx+|v|\int  (D^{1/2}_x(\p_x^2u\,\chi))^2\,dx \\
&= A_{11}(t)+ A_{12}(t).
\end{split}
\end{equation}

Using the previous step \eqref{bo22}, \eqref{FD}, \eqref{INT} and Young's inequality (and recalling the notation $\,\eta^2=\chi'$)  we obtain
\begin{equation}\label{bo26}
\begin{split}
A_{11}(t) & = |v|\, \|D^{1/2}_x(\p_x^2 u\, \eta^2)\|_2^2\\
&\le |v| \, \|D^{1/2}_x(\p_x^2u\,\eta)\|_2^2 \|\eta\|_{\infty}^2+ |v| \|\p_x^2 u\,\eta \|_4^2 \|D_x^{1/2}\eta\|_4^2\\
&\le c\,|v| \,\|D^{1/2}_x(\p_x^2u\,\eta)\|_2^2 +c\,|v|\, \|D^{1/2}_x(\p_x^2 u\,\eta)\|_2\|\p_x^2 u\,\eta\|_2\\
&\le c\,|v|\, \|D^{1/2}_x(\p_x^2u\,\eta)\|_2^2 +c\,|v|\, \|\p_x^2 u\,\eta\|_2^2.
\end{split}
\end{equation}
So after integration in time interval the first term on the right hand side of \eqref{bo26} is bounded. To estimate
the second term on the right hand side  of \eqref{bo26} we use that
$$
\eta_{_{\epsilon, b}}^2=\chi_{_{\epsilon, b}}' \le c \,\chi _{_{\epsilon/5, \epsilon}}.
$$
 Thus we
have
\begin{equation*}
\underset{0\le t\le T}{\sup} \|\p_x^2 u\,\eta\|_2^2\le \underset{0\le t\le T}{\sup} \|\p_x^2 u\, \chi_{_{\epsilon/5, \epsilon}}\|_2^2 \le c
\end{equation*}
by \eqref{bo22} with $(\epsilon,b)=(\epsilon/5,\epsilon)$.

We observe that $A_{12}(t)$ is a multiple of the quantity we are estimating in \eqref{bo24}.

Next we consider $A_2(t)$ in \eqref{bo24}.
\begin{equation}\label{bo27}
\begin{split}
A_2(t) &= - \int D_x(\hil \p_x^4u\,\chi)\,\p_x^2u\,\chi\,dx\\
&=- \int D_x\hil (\p_x^4u\,\chi)\,\p_x^2u\,\chi,dx     +         \int D_x [\hil ; \chi]\,\p_x^4u\,\p_x^2u\,\chi\,dx\\
&= A_{21}(t)+ A_{22}(t).
\end{split}
\end{equation}

From the commutator estimate \eqref{CE} and the previous step \eqref{bo22} it follows that
\begin{equation}\label{bo28}
\begin{split}
|A_{22}(t)| & = \big | \int\hil \p_x [\hil; \chi] \p_x^4 u\, \p_x^2u \,\chi\,dx \big|\\
&\le c\,\|\chi^{(5)}\|_{\infty}\|u(t)\|_2\,\|\p_x^2u(\cdot,t)\,\chi(\cdot+vt)\|_2\\
&\le c\,\|u_0\|_2,
\end{split}
\end{equation}
with $c = c(T,\epsilon, b, v)$ and $t\in[0,T]$.

We turn our attention to $A_{21}(t)$ in \eqref{bo27}.
\begin{equation}\label{bo29}
\begin{split}
A_{21}(t) & = \int \p_x(\p_x^4u\,\chi)\,\p_x^2u\,\chi\,dx\\
& =- \int \p_x^4u\,\chi\,\p_x^3u\,\chi\,dx   -    \int \p_x^4u\,\chi\,\p_x^2u\,\chi'\,dx\\
&= 2 \int (\p_x^3u)^2\,\chi\,\chi'\,dx   -\frac12    \int (\p_x^2u)^2\,(\chi\,\chi')''\,dx\\
&= A_{211}(t) + A_{212}(t).
\end{split}
\end{equation}

From the previous step  \eqref{bo22} with $(\epsilon, b)=(\epsilon/10,\epsilon/2)$
one has that $A_{212}(t)$ is bounded and


\begin{equation}\label{31}
A_{211}(t)\ge 0
\end{equation}
which will provide the smoothing effect after being integrated in time.

So it remains to estimate $A_3=A_3(t)$ in \eqref{bo24},
\begin{equation}\label{bo32}
\begin{split}
A_3(t)&= \int D^{1/2}_x(u\,\p_x^3u\,\chi)\, D^{1/2}_x(\p_x^2u\,\chi)\,dx\\
&\hskip15pt+3\int D^{1/2}_x(\p_x u\,\p_x^2u\,\chi)\, D^{1/2}_x(\p_x^2u\,\chi)\,dx\\
&= A_{31}(t)+ A_{32}(t).
\end{split}
\end{equation}

Thus using \eqref{CL} and \eqref{FD} we have
\begin{equation}\label{bo33}
\begin{split}
|A_{32}(t)| &\le c \|D^{1/2}_x(\partial_xu\chi_{\epsilon/5,\epsilon}\partial_x^2u\chi)\|_2\|D^{1/2}_x(\partial_x^2u\chi)\|_2\\
&\leq c \|\partial_xu\|_{\infty}\|D^{1/2}_x(\partial_x^2u\chi)\|_2^2\\
&\hskip15pt+c\| D^{1/2}_x(\partial_xu\chi_{\epsilon/5,\epsilon})\|_4 \|\partial_x^2u\chi\|_4 \|D^{1/2}_x(\partial_x^2u\chi)\|_2\\
&= \,A_{321}(t) + A_{322}(t).
\end{split}
\end{equation}

The term $A_{321}(t)$ will be handled by Gronwall's inequality and \eqref{bo2} (ii). Now using \eqref{INT}
and Young's inequality one gets
\begin{equation}\label{bo34}
\begin{split}
A_{322}(t) &\le c\|\partial_x(\p_x u\,\chi_{_{\epsilon/5, \epsilon}})\|_2^{3/4}\|\p_xu\,\chi_{_{\epsilon/5, \epsilon}}\|_2^{1/4}\\
&\hskip15pt \times \|D^{1/2}_x(\partial_x^2u\chi)\|_2^{1/2} \|\partial_x^2u\chi\|_2^{1/2}  \|D^{1/2}_x(\partial_x^2u\chi)\|_2\\
&\le c \|\partial_x(\p_x u\,\chi_{_{\epsilon/5, \epsilon}})\|_2^3\|\partial_xu\chi_{\epsilon/5,\epsilon}\|_2\|\partial_x^2u\chi\|_2^2\\
&\hskip15pt + c\,\|D^{1/2}(\partial_x^2u\chi)\|_2^2,
\end{split}
\end{equation}
where the first term on the right hand side of \eqref{bo34} is bounded in time, by our previous step \eqref{bo22}, while the second term in the right hand side of \eqref{bo34} is the quantity to be estimated.

Now we consider $A_{31}(t)$ in \eqref{bo32}. Thus, by \eqref{CL}
\begin{equation}\label{bo35}
\begin{split}
A_{31}(t)&= \int D^{1/2}_x(u\partial_x^3u \chi) \, D^{1/2}_x(\partial_x^2u \chi)\,dx\\
&=c \int u\chi_{_{\epsilon/5,\epsilon}} \, D^{1/2}_x (\partial_x(\partial_x^2u \chi))\, D^{1/2}_x (\partial_x^2u \chi)\,dx\\
&\hskip15pt +c\int [D^{1/2}_x;  u\chi_{_{\epsilon/5,\epsilon}} ] \partial_x(\partial_x^2u \chi) D^{1/2}_x(\partial_x^2u \chi)\,dx\\
&\hskip15pt -c \int D^{1/2}_x(u\partial_x^2u \chi') D^{1/2}_x(\partial_x^2u \chi)\,dx\\
&= A_{311}(t)+A_{312}(t)+A_{313}(t).
\end{split}
\end{equation}

Integration by parts yields
\begin{equation}\label{bo36}
\begin{split}
|A_{311}(t)|&=\big|-\frac12 \int \partial_x(u\chi_{_{\epsilon/5,\epsilon}})\,(D^{1/2}_x (\partial_x^2u \chi))^2\,dx\big|\\
&\le c(\|\partial_xu\|_{\infty}+\|u\|_{\infty})  \int (D^{1/2}_x (\partial_x^2u\, \chi))^2\,dx.
\end{split}
\end{equation}
After applying the Sobolev inequality this term will be  handled by the Gronwall's inequality and \eqref{bo2} (ii).

 From the commutator estimate \eqref{CE2} and the Sobolev inequality one gets
\begin{equation}\label{bo39}
\begin{split}
|A_{312}(t)|&\le \|\widehat{\p_x(u\,\chi_{_{\epsilon/5, \epsilon}})}\|_1\|D^{1/2}_x(\p_x^2u\,\chi)\|_2^2\\
&\le \|\p_x(u\,\chi_{_{\epsilon/5, \epsilon}})\|_{1,2}\|D^{1/2}_x(\p_x^2u\,\chi)\|_2^2\\
&\le \big(\|\p_x(u\,\chi_{_{\epsilon/5, \epsilon}})\|_2+\|\p^2_x u\,\chi_{_{\epsilon/5, \epsilon}}\|_2\\
&\hskip15pt+\|\p_xu\,\chi_{_{\epsilon/5, \epsilon}}'\|_2+\|u\,\chi_{_{\epsilon/5, \epsilon}}''\|_2\big) \|D^{1/2}_x(\p_x^2u\,\chi)\|_2^2\\
&= B(t)\,\|D^{1/2}_x(\p_x^2u\,\chi)\|_2^2.
\end{split}
\end{equation}

Since by \eqref{bo2} (i) and \eqref{bo22} $B(t)$ is 
bounded in the time interval $[0,T]$ and
$\,\|D^{1/2}_x(\p_x^2u\,\chi)\|_2^2$ is the quantity we are estimating,  Gronwall's inequality provides the bound  for
$|A_{312}(t)|$.

Finally, we consider  $A_{313}(t)$ in \eqref{bo35}. Using \eqref{CL} and \eqref{CE2}  yields
\begin{equation}\label{39b}
\begin{split}
|A_{313}(t)|&\le c\|D^{1/2}_x(u\,\chi_{_{\epsilon/5, \epsilon}}\partial_x^2u\,\chi')\|_2\|D^{1/2}_x(\partial_x^2u\,\chi)\|_2\\
&\le c\,\|u\|_{\infty}\|D^{1/2}_x(\partial_x^2u\,\chi')\|_2\|D^{1/2}_x(\partial_x^2u\,\chi)\|_2\\
&\hskip10pt + \|D^{1/2}_x(u\,\chi_{_{\epsilon/5, \epsilon}})\|_4\|\partial_x^2u\,\chi'\|_4\|D^{1/2}_x(\partial_x^2u\,\chi)\|_2\\
&= A_{3131}(t)+A_{3132}(t).
\end{split}
\end{equation}

From an argument similar to the one applied in \eqref{bo34} and \eqref{CL} it follows that
\begin{equation}\label{39c}
\begin{split}
A_{3132} (t)&\le  c\,\|\partial_x(u\,\chi_{_{\epsilon/5, \epsilon}})\|_2^{3/4}\|u\,\chi_{_{\epsilon/5, \epsilon}}\|_2^{1/4}\\
&\hskip10pt\times \|D^{1/2}_x(\partial_x^2u\,\chi')\|_2^{1/2}\|\partial_x^2u\,\chi'\|_2^{1/2}\|D^{1/2}_x(\partial_x^2u\,\chi)\|_2\\
&\le c\|\partial_x(u\,\chi_{_{\epsilon/5, \epsilon}})\|_2^3\|u\,\chi_{_{\epsilon/5, \epsilon}}\|_2\|\partial_x^2u\,\chi_{_{\epsilon/5, \epsilon}}\|_2^2\\
&\hskip10pt+\|D^{1/2}_x(\partial_x^2u\,\chi')\|_2^2+ \|D^{1/2}_x(\partial_x^2u\,\chi)\|_2^2.
\end{split}
\end{equation}

The first term on the last inequality is bounded in time by \eqref{bo22} with $(\epsilon, b)= (\epsilon/5, \epsilon)$. The last term is the quantity we want to estimate
it will be handled using Gronwall's inequality. The second term can be estimated employing \eqref{FD}, \eqref{INT} and Young's inequality, that is,
\begin{equation}\label{bo38}
\begin{split}
\|D^{1/2}_x&(\p_x^2u\,\chi')\|^2_2 =\|D^{1/2}_x(\p_x^2u\,\eta\,\eta)\|_2^2\\
&\le c\|\eta\|_{\infty}^2\|D^{1/2}_x(\p_x^2u\,\eta)\|_2^2+c\|D^{1/2}_x\eta\|_4^2\|\p_x^2u\,\eta\|_4^2\\
&\le c\|D^{1/2}_x(\p_x^2u\,\eta)\|_2^2 +c \|\p_x^2u\,\eta\|_2\|D^{1/2}_x(\p_x^2u\,\eta)\|_2\\
&\le c\|D^{1/2}_x(\p_x^2u\,\eta)\|_2^2+c \|\p_x^2u\,\eta\|_2^2.
\end{split}
\end{equation}
From \eqref{bo22} with $(\epsilon/5, \epsilon)$ instead of $(\epsilon, b)$ the last term is bounded in time $t\in [0,T]$
and the previous one is bounded after integrating in time (see \eqref{bo22}).

Next we deal with the term $A_{3131} (t)$.  Young's inequality gives
\begin{equation}\label{bo39d}
A_{3131} (t)\le c\|u\|_{\infty}^2\|D^{1/2}_x(\partial_x^2u\,\chi)\|_2^2+ \|D^{1/2}_x(\partial_x^2u\,\chi')\|_2^2.
\end{equation}
The first term on the right hand side of \eqref{bo39d} can be handled  using Gronwall's inequality and \eqref{bo2} (i). The last term can be estimated using
the argument in \eqref{bo38}.

Collecting these results we obtain the desired estimate \eqref{bo6} with $m=2$.

\vskip3mm

Following the induction argument we  shall assume that \eqref{bo4} holds for $m\leq j\in\Z^+, \,j\geq 2$, and prove that if \eqref{bo3}  with $x_0=0$ and $m=j+1$ holds, then:

(a) \eqref{bo5} and \eqref{bo6} hold with $x_0=0\,$ and $\,m=j$,

and

(b) \eqref{bo4} holds with $\,x_0=0$ for $m=j+1$.
\vskip3mm
\noindent\underline{Part (a)} :   From the hypothesis \eqref{bo3} with $m=j+1$  and since $u_0\in H^{3/2^+}(\R)$ one gets \eqref{bo5} with $m=j$ by interpolation.

Next,  a familiar argument provides the identity
\begin{equation}\label{bo41}
\begin{split}
&\frac12 \frac{d}{dt} \int (D^{1/2}_x(\p_x^j u\, \chi))^2\,dx -
\underset{A_1}{\underbrace{v\int D^{1/2}_x(\p_x^ju\, \chi')\,D^{1/2}_x(\p_x^j u\, \chi)\,dx}} \\
&-\underset{A_2}{\underbrace{\int D^{1/2}_x(\hil\p_x^{2+j}u\, \chi)\,D^{1/2}_x(\p_x^ju\, \chi)\,dx}}\\
&+\underset{A_3}{\underbrace{\int D^{1/2}_x( \p_x^j(u\p_xu)\,\chi)\,D^{1/2}_x(\p_x^ju\, \chi)\,dx}} = 0.
\end{split}
\end{equation}

First we observe that
\begin{equation}\label{bo42}
\begin{split}
|A_1(t)|&\le |v|\big(\int (D^{1/2}_x(\p_x^ju\,\chi))^2\,dx +\int (D^{1/2}_x(\p_x^ju\,\chi'))^2\,dx \big)\\
&= A_{11}(t)+ A_{12}(t),
\end{split}
\end{equation}
where $A_{11}(t)$ is multiple of the quantity we are estimating in \eqref{bo41} and from \eqref{FD}--\eqref{INT}
we deduce
\begin{equation}\label{bo43}
\begin{split}
A_{12}(t) &= |v| \|D^{1/2}_x(\p_x^ju\, \eta\, \eta)\|_2^2\\
&\le |v|\Big(\|\eta\|_{\infty}^2\|D^{1/2}_x(\p_x^ju\, \eta)\|_2^2
+ \|D^{1/2}_x\eta\|_4^2\|\p_x^j u \, \eta\|_4^2\Big)\\
&\le c\|D_x^{1/2}(\p_x^ju\,\eta)\|^2_2+ c\,\|D^{1/2}_x(\p_x^ju\,\eta)\|_2\|\p_x^ju\,\eta\|_2\\
&\le c\,\|D_x^{1/2}(\p_x^ju\,\eta)\|^2_2+ c\|\p_x^ju\,\eta\|_2^2.
\end{split}
\end{equation}
By \eqref{bo4} with $m=j$ (induction hypothesis) the second term on the right hand side of \eqref{bo43} is bounded, while the first term
is bounded after integration in time.


Next we turn our attention to $A_2(t)$ in \eqref{bo41}
\begin{equation}\label{bo48}
\begin{split}
A_2(t) &= -\int D_x(\hil \p_x^{j+2}u\,\chi)\, \p_x^ju\, \chi\,dx\\
&= -\int D_x\hil(\p_x^{j+2}u\,\chi) \p_x^ju\,\chi\,dx+\int D_x[\hil; \chi] \p_x^{j+2}u\, \p_x^ju\,\chi\,dx\\
&= A_{21}(t) + A_{22}(t).
\end{split}
\end{equation}
where by \eqref{CE} and the conservation of the $L^2$-norm
\begin{equation}\label{bo49}
\begin{split}
|A_{22}(t)|&\le |\int \hil\p_x [\hil; \chi] \,\p_x^{j+2}u\,\p_x^ju \,\chi\,dx| \\
&\le c \|\chi^{(j+3)}\|_{\infty} \|u(t)\|_2 \|\p_x^ju\,\chi\|_2\\
&\le c\|u_0\|^2_2 +c \int (\p_x^ju)^2\,\chi_{_{\epsilon, b}}(x+vt)\,dx \le c
\end{split}
\end{equation}
using that $0\le \chi \le 1$ and the induction hypothesis \eqref{bo4} with $m=j$. Now
\begin{equation}\label{bo50}
\begin{split}
A_{21}(t) &= \int \p_x(\p_x^{j+2}u\,\chi)\, \p_x^ju \chi\,dx\\
&=-\int \p_x^{j+2}u\,\chi\,\p_x^{j+1}u\,\chi-\int \p_x^{j+2}u \,\chi \p_x^ju\,\chi'\,dx\\
&= 2\int (\p_x^{j+1}u)^2\,\chi\,\chi'\,dx+\int \p_x^{j+1}u\,\p_x^ju\,(\chi\,\chi')'\,dx\\
&= 2\, \int (\p_x^{j+1}u)^2\,\chi\,\chi'\,dx -\frac12 \int (\p_x^{j}u)^2\,(\chi\,\chi')''\,dx\\
&= A_{211}(t) + A_{212}(t).
\end{split}
\end{equation}

From the previous step (hypothesis of induction) $m=j$ with $(\epsilon/4, \epsilon)$ instead of $(\epsilon, b)$ one has that $A_{212}(t)$ is
bounded and 
\begin{equation}\label{bo52}
A_{211}(t)\ge 0,
\end{equation}
which yields  the smoothing effect after being integrated in time. So it remains only to consider $A_3$ in \eqref{bo41}.

\begin{equation}\label{bo53}
\begin{split}
A_3(t)&=  \int D^{1/2}_x(u\,\p_x^{j+1}u\, \chi_{_{\epsilon, b}})\, D^{1/2}_x(\p_x^ju\,\chi_{_{\epsilon, b}})\,dx\\
&\hskip15pt + (j+1)\int D^{1/2}_x (\p_xu \p_x^ju\,\chi_{_{\epsilon, b}}) D^{1/2}_x(\p_x^ju\,\chi_{_{\epsilon, b}})\,dx\\
&\hskip15pt + \underset{l=2}{\overset{j-1}{\sum}} \,c_l \int D^{1/2}_x (\p_x^l \,u \p_x^{j+1-l} u\,\chi_{_{\epsilon, b}}) D^{1/2}_x(\p_x^ju\,\chi_{_{\epsilon, b}})\,dx\\
&= A_{31}(t)+ A_{32}(t) +  \underset{l=2}{\overset{j-1}{\sum}} A_{3(l+1)}(t).
\end{split}
\end{equation}

The estimates for $A_{31}$ and $A_{32}$ are similar to those described in \eqref{bo32}--\eqref{bo39} in the case $j=2$. So we restrict ourselves to consider $\{A_{3l}\}_{l=2}^{j-1}$.
\vskip3mm

 We consider first the case $j=3$ where the sum in \eqref{bo53} reduces to the term $A_{33}(t)$.
 \begin{equation}\label{bo54}
 \begin{split}
 |A_{33}(t)| & = c\, |\int D^{1/2}_x(\p_x^2u\,\p_x^2u\, \chi_{_{\epsilon, b}}) D^{1/2}_x(\p_x^3u \,\chi_{_{\epsilon, b}})\,dx|\\
 &\le c\|D^{1/2}_x(\p_x^2u\, \chi_{_{\epsilon/5, \epsilon}}\,\p_x^2u\,\chi_{_{\epsilon, b}})\|_2^2+ \|D^{1/2}_x(\p_x^3u \,\chi_{_{\epsilon, b}})\|_2^2.
 \end{split}
 \end{equation}
 The last term above is the quantity to be estimated so we just need to concentrate in the first term on the right hand side of \eqref{bo54}. Thus
 by \eqref{FD} and \eqref{INT} we deduce that
 \begin{equation}\label{bo55}
 \begin{split}
 \|D^{1/2}_x(\p_x^2u &\chi_{_{\epsilon/5, \epsilon}}\,\p_x^2u \chi_{_{\epsilon, b}})\|_2^2\\
 &\le c\|D^{1/2}_x(\p_x^2u \chi_{_{\epsilon/5, \epsilon}})\|_4\|\p_x^2u \chi_{_{\epsilon, b}}\|_4 \\
&\hskip15pt +c\,\|\p_x^2u \chi_{_{\epsilon/5, \epsilon}}\|_4\|D^{1/2}_x(\p_x^2u \chi_{_{\epsilon, b}})\|_4\\
 &\le c \|\p_x(\p_x^2u \chi_{_{\epsilon/5, \epsilon}})\|_2^{3/4}\|\p_x^2u\, \chi_{_{\epsilon/5, \epsilon}}\|_2^{1/4}\\
&\hskip15pt \times \|D^{1/2}_x(\p_x^2u \chi_{_{\epsilon, b}})\|_2^{1/2}\|\p_x^2u \chi_{_{\epsilon, b}}\|_2^{1/2} \\
 &\hskip15pt + c\,\|D^{1/2}_x(\p_x^2u \chi_{_{\epsilon/5, \epsilon}})\|_2^{1/2}\|\p_x^2u \chi_{_{\epsilon/5, \epsilon}}\|_2^{1/2}\\
&\hskip15pt \times\|\p_x(\p_x^2u \chi_{_{\epsilon, b}})\|_2^{3/4}\|\p_x^2u \chi_{_{\epsilon, b}}\|_2^{1/4}\\
 &\equiv A_{331},
 \end{split}
 \end{equation}
 where all the terms in $A_{331}(t)$ involve at most derivatives of order three which are  bounded in $t\in[0,T]$  by hypothesis of induction (and the fact that $0\le \chi \le 1$).

 It is clear from the argument given in \eqref{bo54}--\eqref{bo55} for the case $j=3$ that the general case $j\ge 4$ follows by using the
 same method which combines  the inequalities \eqref{FD}, \eqref{INT}, Sobolev embedding theorem and induction hypothesis.

 Inserting all the estimates above in \eqref{bo41}  one gets that \eqref{bo6} holds for $m=j$.

 This completes the proof of the first step, part (a), of our inductive argument.

 Next we consider the next step in our induction argument:

 \vspace{3mm}
 \noindent\underline{Part (b)} : So we assume that \eqref{bo3} holds with $m=j+1$, our induction hypothesis, i.e., for $k=2,...,j$
 \begin{equation}\label{bo55b}
 \begin{split}
 \underset{0\le t\le T}{\sup} &\int (\p_x^{k}u)^2 \chi_{_{\epsilon, b}}(x+vt)\,dx\\
 &+\int\limits_0^T \int (D^{1/2}_x(\p_x^{k}u \eta_{_{\epsilon, b}}))^2 \,dxdt  < c,
 \end{split}
 \end{equation}
 and recall that in part (a) of our argument we have proven \eqref{bo6} with $m=j$.
 A familiar argument yields the identity
 \begin{equation}\label{bo55c}
\begin{split}
&\frac12 \frac{d}{dt} \int (\p_x^{j+1}u)^2\, \chi_{_{\epsilon, b}}(x+vt)\,dx -
\underset{A_1}{\underbrace{v\int (\p_x^{j+1}u)^2\, \chi_{_{\epsilon, b}}'(x+vt)\,dx}} \\
&-\underset{A_2}{\underbrace{\int \hil\p_x^{j+1}\p_x^2 u\, \p_x^{j+1}u\, \chi_{_{\epsilon, b}}(x+vt)\,dx}}\\
&+\underset{A_3}{\underbrace{\int \p_x^{j+1}(u\p_xu)\,\p_x^{j+1}u\,\chi_{_{\epsilon, b}}(x+vt)\,dx\,dx}} = 0.
\end{split}
\end{equation}

From our previous step, part (a),  \eqref{bo6} holds with $m=j$. Therefore we have that for $\epsilon'>0$ and $b>5\epsilon'$
\begin{equation}\label{56}
\int\limits_0^T\int (\p_x^{j+1}u)^2 \,\chi_{_{\epsilon', b}}'\,\chi_{_{\epsilon', b}}(x+vt)\,dxdt< c.
\end{equation}

Using \eqref{CL} after integrating in the time interval $[0,T]$ we have that
\begin{equation}\label{bo57}
\int\limits_0^T|A_1(t)|\,dt \le c = c(\epsilon, b, v, T).
\end{equation}

Next we consider $A_2(t)$ in \eqref{bo55c}. Thus
\begin{equation}\label{bo58}
\begin{split}
A_2&=\int \hil \p_x^{j+2}u\,\p_x^{j+2}u\,\chi\,dx +\int \hil \p_x^{j+2}u\,\p_x^{j+1}u \chi'\,dx\\
&=A_{21}+A_{22},
\end{split}
\end{equation}
where
\begin{equation}\label{bo59}
\begin{split}
A_{21}&=-\int \p_x^{j+2}u\,\hil(\p_x^{j+2}u\,\chi)\,dx \\
&=-\int \p_x^{j+2}u\,\hil\p_x^{j+2}u\,\chi\,dx-\int \p_x^{j+2}u\,[\hil;\chi] \p_x^{j+2}u\\
&= -A_{21}+ (-1)^{j+3} \int u \p_x^{j+2}[\hil; \chi]\p_x^{j+2}u\,dx.
\end{split}
\end{equation}
Thus
\begin{equation}\label{bo60}
A_{21}(t)=\frac12 (-1)^{j+3} \int u \p_x^{j+2}[\hil; \chi]\p_x^{j+2}u\,dx.
\end{equation}
So by the commutator estimate  \eqref{CE}
 and the conservation law
 \begin{equation}\label{bo61}
 \begin{split}
 |A_{21}(t)| &\le c\|\chi^{2j+4}\|_{\infty}\|u(t)\|_2^2\le c_j\,\|u_0\|_2^2.
 \end{split}
 \end{equation}

 To estimate $A_{22}$, recalling that $\eta^2=\chi'$, we write
 \begin{equation}\label{bo62}
 \begin{split}
 A_{22} & = \int \hil \p_x^{j+2}u\,\p_x^{j+1}u\,\eta^2(x+vt)\,dx\\
 &=\int  \hil ( \p_x^{j+2}u\,\eta)\,\p_x^{j+1}u\,\eta \,dx
  -\int [\hil; \eta] \p_x^{j+2}u\,\p_x^{j+1}u\,\eta\,dx\\
 &= A_{221}+A_{222}.
 \end{split}
 \end{equation}

Now
\begin{equation}\label{bo63}
\begin{split}
A_{221}&= \int \hil\p_x (\p_x^{j+1}u\,\eta)\,\p_x^{j+1}u\,\eta\,dx -\int \hil(\p_x^{j+1}u\,\eta') \p_x^{j+1}u\,\eta\,dx\\
&=\int D^{1/2}_x(\p_x^{j+1}u\,\eta) \,D^{1/2}_x(\p_x^{j+1}u\,\eta)\\
&\hskip15pt- \int \hil(\p_x^{j+1}u\,\eta') \p_x^{j+1}u\,\eta\,dx.
\end{split}
\end{equation}

Notice that the first term on the right hand side of \eqref{bo63} is positive (and will provide the smoothing effect)
and the second one can be bounded by
\begin{equation}\label{bo64}
\|\p_x^{j+1}u\,\eta'\|_2^2+\|\p_x^{j+1}u \eta\|_2^2
\end{equation}
which after integration in time is bounded using part (a) in our argument, i.e. \eqref{bo6} with $m=j$, and the claim
following   \eqref{2.8} together with \eqref{2.5} and the estimates \eqref{CL}.

Similarly,  $A_{222}$ can be estimated with the aid of \eqref{CE}.

So it remains only to consider $A_3$. Thus
\begin{equation}\label{bo65}
\begin{split}
A_3 &= \int u\,\p_x^{j+2}u\,\p_x^{j+1}u\,\chi\,dx+ (j+2)\int \p_xu\,(\p_x^{j+1}u)^2\,\chi \,dx\\
&\hskip15pt +\underset{k=2}{\overset{j}{\sum}} c_k\int \p_x^ku\,\p_x^{j+2-k}(\p_x^{j+1}u)\,\chi\,dx \hskip10pt (j\ge 2)\\
&= A_{30}+A_{31}+ \underset{k=2}{\overset{j}{\sum}} A_{3k}.
\end{split}
\end{equation}

Thus
\begin{equation}\label{bo66}
\begin{split}
A_{30}&=-\frac12 \int \p_x u \p_x^{j+1}u\,\p_x^{j+1}u\,\chi\,dx-\int u\p_x^{j+1}u\,\p_x^{j+1}u\,\chi'\,dx\\
& = A_{301}+A_{302}
\end{split}
\end{equation}
with
\begin{equation}\label{bo67}
|A_{301}(t)| \le \|\p_xu(t)\|_{\infty} \int (\p_x^{j+1} u)^2\,\chi\,dx
\end{equation}
where the last term is the quantity to be estimate which will be handled in Gronwall's inequality using \eqref{bo2} (ii) and
\begin{equation}\label{bo68}
|A_{302}(t)| \le \|u(t)\|_{\infty} \int (\p_x^{j+1}u)^2\, \chi'\,dx
\end{equation}
which after integration in time is bounded using part (a) of our
argument, i.e. \eqref{bo6} with $m=j$ and  \eqref{CL}.

The estimate for $A_{31}(t)$ is similar to that one described in \eqref{bo67} for $A_{301}(t)$.

Next, we consider the case $j=2$ where  in \eqref{bo65}  the term $A_{32}$ appears, i.e.
\begin{equation}\label{bo69}
\begin{split}
|A_{32}(t)| &\le |c_2 \int \p_x^2u \p_x^2 u (\p_x^3u)\,\chi\,dx|\\
&=|-\frac{c_2}{3} \int \p_x^2u \p_x^2u \p_x^2u \,\chi'\,dx|\\
&=|-\frac{c_2}{3} \int \p_x^2 u \eta (\p_x^2u)^2\,\eta\,dx|\\
&\le c\|\p_x^2u \,\eta\|_{\infty}\int (\p_x^2u)^2\,\eta\,dx.
\end{split}
\end{equation}

Since for the case $j=2$ we have that
\begin{equation}\label{bo70}
\int (\p_x^2u)^2\, \eta_{_{\epsilon,b}}(x+vt)\,dx\le \int (\p_x^2u)^2\,\chi_{_{\epsilon/5, \epsilon}}(x+vt)\,dx
\end{equation}
which is bounded in the time interval $[0,T]$. Now
\begin{equation}\label{bo71}
\begin{split}
\|\p_x^2u\,\eta\|_{\infty} &\le c \|\p_x(\p_x^2u\,\eta)\|_2^{1/2}\| \p_x^2u\,\eta\|_2^{1/2}\\
&\le c \big(\|\p_x^3u\,\eta\|_2^{1/2}+\|\p_x^2u\,\eta'\|_2^{1/2}\big)\|\p_x^2u\,\eta\|_2^{1/2}.
\end{split}
\end{equation}

From the previous step we have that
\begin{equation*}
\underset{0\le t\le T} {\sup}\big(\|\p_x^2u\,\eta\|_2+\|\p_x^2u\,\eta'\|_2\big)\le c.
\end{equation*}
So
\begin{equation}\label{bo72}
\|\p_x^2u\,\eta\|_{\infty}\le c+ \|\p_x^3u\,\eta'\|_2^2
\end{equation}
which is bounded after integrating in time from the part (a), \eqref{bo6} with $m=j$ (in this case $j=3$), in our argument.

If $\,j\ge 2$ the new terms in \eqref{bo65} can be handled in a similar manner combining Sobolev theorem, the
induction hypothesis  and Gronwall's inequality. This completes the induction argument.

 To justify the above formal computations we shall follow the following standard argument.

 Consider data $u_0^{\tau}=\rho_{\tau} \ast u_0$ with $\rho\in C^{\infty}_0(\R)$, $\text{supp}\,\rho\in (-1,1)$, $\;\rho\ge 0$, $\;\;\displaystyle\int \rho(x)\,dx=1$ and
 \begin{equation*}
  \rho_{\tau}(x)=\frac{1}{\tau} \rho\big(\frac{x}{\tau}\big), \;\;\tau>0.
 \end{equation*}

For $\tau>0$ consider  the  solutions $u^{\tau}$ of the IVP \eqref{bo1} with data $u_0^{\tau}$  where $(u^{\tau})_{\tau>0}\subseteq C([0,T]: H^{\infty}(\R))$.

 Using the continuous dependence of the solution upon the data we have that
\begin{equation}
\label{123}
\sup_{t\in [0,T]}\;\| u^{\tau}(t)-u(t)\|_{3/2,2}\,\downarrow 0\;\;\;\;\text{as}\;\;\;\;\tau\,\downarrow 0.
\end{equation}
 Applying our argument  to the smooth solutions $u^{\tau}(\cdot,t)$ one gets that
\begin{equation}\label{124}
\sup_{[0,T]} \int (\partial^m_x u^{\tau})^2\, \chi_{\epsilon,b}(x+vt)\,dx\le c_0
\end{equation}
for any $\epsilon>0$, $b\ge 5\epsilon$, $v>0$, $c_0= c_0(\epsilon;b;v)>0$  but independent of $\tau>0$
since for $0<\tau<\epsilon$
$$
(\partial_x u_0^{\tau})^2\chi_{\epsilon,b}(x)=(\partial_x(\rho_{\tau} \ast u_0))^2 \chi_{\epsilon,b}(x)=(\rho_{\tau}\ast\partial_xu_0 1_{[0,\infty)})^2 \chi_{\epsilon,b}(x).
$$
Combining \eqref{123} and \eqref{124} and a weak compactness argument one gets that
\begin{equation}\label{125}
\sup_{[0,T]} \int (\partial_x u)^2\, \chi_{\epsilon,b}(x+vt)\,dx\le c_0
\end{equation}
which is the desired result. A similar argument provides the estimate for the second term in the left hand side of \eqref{bo4}.

This completes the proof of Theorem A.

\vskip.2in

\section*{Acknowledgments} P. I. was supported by Universidad Na\-cio\-nal de Co\-lom\-bia-Me\-de\-ll\'in. F. L. was partially supported
by CNPq and FAPERJ/Brazil. G. P. was  supported by a NSF grant  DMS-1101499.

\end{document}